\documentclass[a4paper,onecolumn,10pt]{amsart}

\usepackage{amsmath}
\usepackage{amssymb}  
\usepackage{amsthm}
\usepackage{amsfonts}
\usepackage{graphicx}
\usepackage{cancel}
\usepackage{bbm}
\usepackage{url}
\usepackage{enumerate} 
\usepackage{ upgreek }
\usepackage{dsfont}
\usepackage{color}
\usepackage{subcaption}
\usepackage{comment}

\usepackage[normalem]{ulem}

\usepackage{makecell}
\usepackage{diagbox}
\usepackage{wrapfig}
\usepackage{rotating}
\usepackage{tabularx}

\usepackage{hyperref}
\hypersetup{
    colorlinks=true,
    linkcolor=blue,
    citecolor=red,
    filecolor=magenta,      
    urlcolor=cyan,
}

\newcommand{\tr}{\textnormal{Tr}}

\newcommand{\braket}[2]{\langle #1 , #2 \rangle}

\def\beq{\begin{equation}}
\def\eeq{\end{equation}}
\def\bq{\begin{quote}}
\def\eq{\end{quote}}
\def\ben{\begin{enumerate}}
\def\een{\end{enumerate}}
\def\bit{\begin{itemize}}
\def\eit{\end{itemize}}

\def\ra{\rightarrow}

\def\lset{\lbrace}
\def\rset{\rbrace}

\def\r|{\right|}

\def\ident{\textnormal{id}}

\newcommand\C{\mathbf{C}}

\newcommand\R{\mathbf{R}}

\newcommand\N{\mathbf{N}}
\newcommand\E{\mathbf{E}}

\newcommand{\scalar}[2]{\langle#1,#2\rangle}

\theoremstyle{plain}
\newtheorem{thm}{Theorem}[section]
\newtheorem{lem}[thm]{Lemma}
\newtheorem{cor}[thm]{Corollary}
\newtheorem{prop}[thm]{Proposition}

\newtheorem{conj}[thm]{Conjecture}

\theoremstyle{definition}

\newcommand{\iy}{\infty}
\newcommand{\e}{\varepsilon}
\renewcommand{\epsilon}{\varepsilon} 
\renewcommand{\leq}{\leqslant}
\renewcommand{\geq}{\geqslant}

\begin{document}

\title{Limit formulas for norms of tensor power operators}

\author{Guillaume Aubrun}
\address{\small{Institut Camille Jordan, Universit\'{e} Claude Bernard Lyon 1, 43 boulevard du 11 novembre 1918, 69622 Villeurbanne cedex, France}}
\email{aubrun@math.univ-lyon1.fr}
\author{Alexander M\"uller-Hermes}
\address{\small{Department of Mathematics, University of Oslo, P.O. box 1053, Blindern, 0316 Oslo, Norway}}
\email{muellerh@math.uio.no, muellerh@posteo.net}


\date{\today}
\begin{abstract}
Given an operator $\phi:X\ra Y$ between Banach spaces, we consider its tensor powers $\phi^{\otimes k}$ as operators from the $k$-fold injective tensor product of~$X$ to the $k$-fold projective tensor product of $Y$. We show that after taking the $k$th root, the operator norm of $\phi^{\otimes k}$ converges to the $2$-dominated norm $\gamma^*_2(\phi)$, one of the standard operator ideal norms. 
\end{abstract}
\maketitle

\section{Introduction}

Let $\phi : X \to Y$ denote a bounded operator between Banach spaces. We denote by $\left\|  \phi^{\otimes k}  \right\|_{\e \to \pi}$ the operator norm of the map~$\phi^{\otimes k}$ when the domain is $X^{\otimes_\e k}$, the algebraic tensor product equipped with the injective norm, and the range is $Y^{\otimes_\pi k}$, the algebraic tensor product equipped with the projective norm. Our main contribution is to prove the formula
\begin{equation} \label{eq:single-letter-formula} \lim_{k \to \iy} \left\| \phi^{\otimes k}  \right\|_{\e \to \pi}^{1/k} = \gamma^*_2(\phi). \end{equation}
Here, $\gamma_2^*$ is the $2$-dominated norm, which is the adjoint operator ideal norm to~$\gamma_2$, the factorization norm through a Hilbert space. We take the convention that the norm of an unbounded operator equals $+\iy$. In particular, the limit in~\eqref{eq:single-letter-formula} is finite if and only if the operator $\phi$ is $2$-dominated.

The process of taking the limit such as in \eqref{eq:single-letter-formula} is called a \emph{regularization} in quantum information theory (see for example~\cite[Section 8]{watrous2018theory}). It is often unclear how to compute such regularizations directly, and a \emph{single-letter formula} such as~\eqref{eq:single-letter-formula}, i.e., a formula only depending on the map~$\phi$ itself and not its tensor powers, is highly desirable. In our situation, it allows us to answer a number of questions left open in our previous paper~\cite{paperB}. The classes of operators $\phi$ for which $\|\phi^{\otimes k}\|_{\e \to \pi} < \iy$ for a fixed $k$ have previously been studied by K.~John (see for example \cite{John88}).

As a key step towards proving \eqref{eq:single-letter-formula}, we show another single-letter formula. Consider a bounded operator $\phi : X \to H$ between a Banach space $X$ and a Hilbert space $H$. Denote by $\| \phi^{\otimes k}\|_{\e \to h}$ the norm of the map $\phi^{\otimes k}$ from $X^{\otimes_\e k}$ to $H^{\otimes_h k}$, the space $H^{\otimes k}$ equipped with the Hilbert space tensor product structure. We prove the formula
\begin{equation} \label{eq:single-letter-formula-pi2} \lim_{k \to \iy} \left\| \phi^{\otimes k}  \right\|_{\e \to h}^{1/k} = \pi_2(\phi), \end{equation}
where $\pi_2$ is the $2$-summing norm.

\section{Preliminaries on tensor norms and operator ideals}

Let $X$ and $Y$ be real or complex Banach spaces. The closed unit ball of $X$ is denoted by $B_X$. Given a linear operator $\phi : X \to Y$, we denote its norm by $\|\phi : X \to Y\|$, and we write simply $\|\phi\|$ when there is no ambiguity on the norm used in the domain and in the range. We take the convention that $\|\phi\|=+\iy$ if $\phi$ is not bounded. Similar conventions apply to all the norms introduced later. 

\subsection{Tensor norms}

Let $X\otimes Y$ denote the algebraic tensor product of the Banach spaces $X$ and $Y$. For $z\in X\otimes Y$, the \emph{projective norm} is given by
\[
\|z\|_\pi = \inf\lset \sum^n_{i=1} \|x_i\|\|y_i\|~:~n\in\N,\, z=\sum^n_{i=1} x_i\otimes y_i \rset ,
\]
and the \emph{injective norm} is given by 
\[
\|z\|_\epsilon = \sup\lset |\braket{\alpha\otimes \beta}{z}|~:~\alpha\in B_{X^*},\, \beta\in B_{Y^*}\rset .
\]
We denote by $X \otimes_\pi Y$ and $X \otimes_\e Y$ the normed spaces $(X \otimes Y, \|\cdot\|_\pi)$ and $(X \otimes Y, \|\cdot\|_\e)$. The injective and projective tensor products are associative and allow us to define inductively $X^{\otimes_\e k}$ and $X^{\otimes_\pi k}$ for any integer $k \geq 1$. If $H_1$, $H_2$ are Hilbert spaces, we denote by $H_1 \otimes_h H_2$ the space $H_1 \otimes H_2$ with the canonical inner product, and by $H_1^{\otimes_h k}$ the $k$-fold Hilbert space tensor product. 

We often use the following property called the metric mapping property: given bounded operators $\phi_1 : X_1 \to Y_1$ and $\phi_2 : X_2 \to Y_2$, then for $\alpha = \pi$ or $\alpha=\e$ (or $\alpha=h$ if the involved spaces are Hilbert spaces), the map
\[ \phi_1 \otimes \phi_2 : X_1 \otimes_\alpha X_2 \to Y_1 \otimes_\alpha Y_2\]
is bounded and satisfies $\|\phi_1 \otimes \phi_2\| = \|\phi_1\| \cdot \|\phi_2\|$.

\subsection{2-summing norm}

A bounded operator $\phi : X \to Y$ is said to be $2$-summing if there exists a constant $C$ such that the inequality
\[ \left( \sum_{i} \|\phi(x_i)\|^2 \right)^{1/2} \leq C \sup_{f \in B_{X^*}} \left( \sum_i |f(x_i)|^2 \right)^{1/2} \]
holds for every finite subset $(x_i)$ of $X$. The smallest such $C$ is called the $2$-summing norm of $\phi$, denoted by $\pi_2(\phi)$. If an operator $\phi$ is not $2$-summing, it is understood that $\pi_2(\phi)=+\iy$.

The class of $2$-summing operators between Hilbert spaces coincides with the class of Hilbert--Schmidt operators, and we have $\pi_2(\phi) = \mathrm{hs}(\phi)$, the Hilbert-Schmidt norm, for every operator~$\phi$ between Hilbert spaces (see~\cite[Theorem 4.10]{DJT95}). We also use the following property of the $2$-summing norm 
\cite[Proposition~2.7]{DJT95}: given a Hilbert space~$H$ and a bounded operator $\phi : X \to H$, we have 
\begin{equation} \label{eq:pi2-as-a-sup} \pi_2(\phi) = \sup \{ \mathrm{hs}(\phi \lambda) ~:~ \lambda : H \to X,\ \|\lambda\| \leq 1 \}.\end{equation}
The Pietsch factorization theorem (see \cite[Theorem 2.13 and Corollary~2.16]{DJT95}) shows that a bounded operator $\phi : X \to Y$ is $2$-summing if and only if it factors as 
\begin{equation}\label{equ:2NucFact} \phi : X \stackrel{\alpha}{\longrightarrow} L^{\iy}(\mu) \stackrel{i_\mu}{\longrightarrow} L^2(\mu) \stackrel{\beta}{\longrightarrow} Y , \end{equation}
where $(\Omega,\mu)$ is a probability space, $\alpha$, $\beta$ are bounded and $i_\mu : L^{\iy}(\mu) \to L^2(\mu)$ is the canonical embedding. Moreover, $\pi_2(\phi)$ equals the infimum of the products $\|\alpha\| \cdot \|\beta\|$ over factorizations as in \eqref{equ:2NucFact}. 

\subsection{The Hilbert space factorization norm and its trace dual}

A bounded operator $\phi:X\ra Y$ factoring as
\begin{equation}\label{equ:HFact} \phi : X \stackrel{\phi_1}{\longrightarrow} H \stackrel{\phi_2}{\longrightarrow} Y , \end{equation}
with a Hilbert space $H$ and bounded operators $\phi_1:X\ra H$ and $\phi_2:H\ra Y$ is said to factor through a Hilbert space. The Hilbert space factorization norm of such an operator is defined as
\[
\gamma_2(\phi) = \inf\lset \|\phi_1\| \cdot \|\phi_2 \|~:~\phi=\phi_2\phi_1\rset ,
\] 
where the infimum is over factorizations as in \eqref{equ:HFact}. If an operator $\phi$ does not factor through a Hilbert space, it is understood that $\gamma_2(\phi)=+\infty$.

A bounded operator $\phi : X \to Y$ is said to be $2$-dominated if it factors as in \eqref{equ:HFact} with a Hilbert space $H$ and with $\phi_1:X\ra H$ and $\phi^*_2:Y^*\ra H^*$ being $2$-summing. The $2$-dominated norm of such an operator is defined as
\[
\gamma_2^*(\phi) = \inf\lset \pi_2(\phi_1)\pi_2(\phi^*_2)~:~\phi=\phi_2\phi_1\rset ,
\] 
where the infimum is over factorizations as in \eqref{equ:HFact}. Again, if an operator $\phi$ is not $2$-dominated it is understood that $\gamma_2^*(\phi)=+\infty$.

We first state the duality which links $\gamma_2$ and $\gamma_2^*$ (see \cite[Theorem 7.11.]{DJT95}) in the case of a linear operator $\phi:X \to Y$ between finite-dimensional normed spaces. In that case, we have 
\[ \gamma_2^*(\phi) = \sup \{ |\tr (\phi \psi)| ~:~ \psi : Y \to X,\ \gamma_2(\psi) \leq 1 \} .\]
Using cyclicity of the trace, this formula can be rewritten as
\begin{equation}
\label{eq:trace-duality}
\gamma_2^*(\phi) = \sup \{ |\tr (\beta \phi \alpha)|
\} \end{equation}
where the supremum is over finite-dimensional Hilbert spaces $H$ and maps $\alpha : H \to X$ and $\beta : Y \to H$ of norm less than $1$.

For bounded operators between general Banach spaces, the formula \eqref{eq:trace-duality} must be modified (see~\cite[Section 10]{DJT95}). For a bounded operator $\phi : X \to Y$ between Banach spaces $X$ and $Y$ we have
\begin{equation}
\label{eq:trace-duality-infdim}
\gamma_2^*(\phi)=\sup\lset |\tr(\psi_1\beta\phi\alpha \psi_2)| \rset
\end{equation}
where the supremum is over finite-dimensional normed spaces $E$ and $F$, finite-dimensional Hilbert spaces $H$, and operators $\alpha:E\ra X$, $\beta:Y\ra F$, $\psi_1:F\ra H$ and $\psi_2:H\ra E$ all with norm less than $1$. 

\section{Proof of the main identities in finite dimensions}

In this section, we present the proof of the main identities in the case of finite-dimensional spaces. The infinite-dimensional case, which requires minor modifications, is covered in Section \ref{sec:infinite-dimension}.

\subsection{Injective-to-Hilbertian regularization}

We start by proving a stability property of the $2$-summing norm under tensor products.

 \begin{prop} \label{prop:pi2-tensorstable}
 Let $X_1$, $X_2$ be normed spaces, $H_1$, $H_2$ be Hilbert spaces, all of finite dimension. For every operators $\phi_1:X_1 \to H_1$ and $\phi_2 : X_2 \to H_2$, the operator
 \[
 \phi_1\otimes \phi_2: X_1\otimes_\epsilon X_2\ra H_1\otimes_h H_2
 \]
 satisfies $\pi_2(\phi_1 \otimes \phi_2)=\pi_2(\phi_1)\pi_2(\phi_2)$.
 \end{prop}

The choice of tensor norms on domain and range is crucial here: For instance, the $2$-summing norm is not multiplicative, when both the domain and the range are equipped with the projective tensor norms~\cite[Corollary 1, Section 34.1]{defant1992tensor}. 

\begin{proof}
For $j \in \{1,2\}$, consider factorizations $\phi_j=\beta_j i_{\mu_j} \alpha_j$ given as
\[ \phi_j : X_j \stackrel{\alpha_j}{\longrightarrow} L^{\iy}(\mu_j) \stackrel{i_{\mu_j}}{\longrightarrow} L^2(\mu_j) \stackrel{\beta_j} {\longrightarrow} H_j . \]
Let $\mu = \mu_1 \otimes \mu_2$. We identify the completion of $L^2(\mu_1) \otimes_h L^2(\mu_2)$ with $L^2(\mu)$. The canonical map $\kappa : L^{\iy}(\mu_1) \otimes_\e L^{\iy}(\mu_2) \to L^{\iy}(\mu)$ is isometric.

Consider the factorization  of $\phi_1 \otimes \phi_2$ as
\[ 
X_1 \otimes_\e X_2 \stackrel{\alpha_1 \otimes \alpha_2}{\longrightarrow} 
L^{\iy}(\mu_1) \otimes_\e L^{\iy}(\mu_2) \stackrel{\kappa}{\longrightarrow} L^{\iy}(\mu) \stackrel{i_{\mu}}{\longrightarrow} L^2(\mu) \stackrel{\beta_1 \otimes \beta_2}{\longrightarrow} H_1 \otimes_h H_2 . \]
By the metric mapping property we have
\[
\| \alpha_1 \otimes \alpha_2 : X_1 \otimes_\e X_2 \to L^{\iy}(\mu_1) \otimes_\e L^{\iy}(\mu_2) \| = \|\alpha_1 \| \cdot \|\alpha_2\|, 
\]
and
\[
 \| \beta_1 \otimes \beta_2 : L^2(\mu) \to H_1 \otimes_h H_2 \| = \| \beta_1 \| \cdot \|\beta_2\|.
\]
We conclude that $\pi_2(\phi_1 \otimes \phi_2) \leq \|\alpha_1\| \cdot \|\alpha_2\| \cdot \|\beta_1\| \cdot \|\beta_2\|$. The inequality $\pi_2(\phi_1 \otimes \phi_2) \leq \pi_1(\phi_1)\pi_2(\phi_2)$ follows by taking the infimum over all such factorizations.

For the converse inequality, consider for $j\in\lset 1,2\rset$ operators $\lambda_j : H_j \to X_j$ such that $\|\lambda_j\| \leq 1$ and set $\lambda = \lambda_1 \otimes \lambda_2$. We have
\[ \| \lambda : H_1 \otimes_h H_2 \to X_1 \otimes_\e X_2 \| \leq \| \lambda : H_1 \otimes_\e H_2 \to X_1 \otimes_\e X_2 \| = \|\lambda_1\| \cdot \|\lambda_2\| \leq 1, \]
using $\|\cdot\|_\e \leq \|\cdot\|_h$ and the metric mapping property. It follows that
\[ 
\mathrm{hs} (\phi_1 \lambda_1) \mathrm{hs}(\phi_2 \lambda_2)
= \mathrm{hs}( (\phi_1 \lambda_1) \otimes (\phi_2 \lambda_2) )
=\mathrm{hs}( (\phi_1 \otimes \phi_2) \lambda) \leq \pi_2(\phi_1 \otimes \phi_2) ,
\]
by \eqref{eq:pi2-as-a-sup}. The inequality $\pi_2(\phi_1) \pi_2(\phi_2) \leq \pi_2(\phi_1 \otimes \phi_2)$ follows by taking the supremum over operators $\lambda_1$ and $\lambda_2$.
\end{proof}

By Proposition \ref{prop:pi2-tensorstable} and since the operator norm is smaller than the $2$-summing norm we have
\begin{equation} \label{eq:tau_eh_upperbound} \limsup_{k\ra \infty}\|\phi^{\otimes k} \|_{\e \to h}^{1/k} \leq \limsup_{k\ra \infty}\pi_2(\phi^{\otimes k} : X^{\otimes_\e k} \to H^{\otimes_h k} )^{1/k} = \pi_2(\phi),\end{equation}
for every operator $\phi:X \to H$. To prove the reverse inequality in \eqref{eq:single-letter-formula-pi2}, we start in the setting of Hilbert spaces. 

\begin{prop} \label{p:Lemma2EpsilonH}
For every finite-dimensional Hilbert space $H$ and every operator $\phi : H \to H$, we have
\[ \lim_{k \to \iy} \|\phi^{\otimes k}\|_{\e \to h}^{1/k} = \mathrm{hs}(\phi) = \pi_2(\phi). \]
\end{prop}

\begin{proof}
We have $\pi_2(\phi) = \mathrm{hs}(\phi)$ (see~\cite[Theorem 4.10]{DJT95}). A standard probabilistic construction (see Appendix) shows that for every $k \geq 2$ there exists a tensor $z_k \in H^{\otimes k}$ such that
\[
\|z_k\|_{\epsilon}\leq C\sqrt{k\log(k)}, 
\]
and
\[
\|\phi^{\otimes k}(z_k)\| \geq \frac{1}{\sqrt{3}} \mathrm{hs}(\phi)^{k} .
\]
Here $C$ only depends on $\dim(H)$. We conclude that 
\[
\|\phi^{\otimes k}\|_{\epsilon\ra h}^{1/k} \geq \mathrm{hs}(\phi) \left(\frac{1}{C\sqrt{3 k \log k}}\right)^{1/k},
\]
for every $k\in\N$. Comparing with \eqref{eq:tau_eh_upperbound}, it follows that
\[
\lim_{k \to \iy} \|\phi^{\otimes k}\|_{\e \ra h}^{1/k} = \mathrm{hs}(\phi).
\qedhere
\]
\end{proof}

Next, we can treat the general case.

\begin{thm} \label{thm:Main1}
Let $X$ denote a normed space and $H$ a Hilbert space, both finite-dimensional. For every operator $\phi : X \to H$, we have
\[ \lim_{k \to \iy} \|\phi^{\otimes k}\|^{1/k}_{\e \to h} = \pi_2(\phi). \]
\end{thm}

\begin{proof}
Let $\phi:X\to H$ be $2$-summing and consider a bounded operator $\lambda : H \to X$ with $\| \lambda  \|\leq 1$. By Proposition \ref{p:Lemma2EpsilonH} and the metric mapping property we have
\[
\liminf_{k \to \iy} \|\phi^{\otimes k}\|^{1/k}_{\e \to h}\geq
\liminf_{k \to \iy} \|(\phi \lambda)^{\otimes k}\|^{1/k}_{\e \to h}
 = \pi_2(\phi\lambda),
\]
By \eqref{eq:pi2-as-a-sup}, taking the supremum over $\lambda$ shows that 
\[ \liminf_{k \to \iy} \|\phi^{\otimes k}\|^{1/k}_{\e \to h}
 \geq \pi_2(\phi) . \]
Using \eqref{eq:tau_eh_upperbound} finishes the proof.
\end{proof}

By the duality of the injective and projective tensor norms, we obtain the following corollary to Theorem~\ref{thm:Main1} concerning the operator norms $\|\phi^{\otimes k}\|_{h \to \pi}$ of operators $\phi:H\ra X$ where the domain $H^{\otimes k}$ is equipped with the Hilbert space tensor product structure and the range $X^{\otimes k}$ with the projective tensor norm. 

\begin{cor} \label{cor:Main12}
Let $H$ be a finite-dimensional Hilbert space and $Y$ be a finite-dimensional Banach space. For every operator $\phi : H \to Y$, we have
\[ \lim_{k \to \iy} \|\phi^{\otimes k}\|_{h \to \pi}^{1/k} = \pi_2(\phi^*). \]
\end{cor}

\subsection{Injective-to-projective regularization} 

Consider an operator $\phi:X\ra Y$ factoring as $\phi=\phi_2\phi_1$ through a Hilbert space $H$ with operators $\phi_1:X\ra H$ and $\phi_2:H\ra Y$. By submultiplicativity of the operator norm, we have
\[ \| \phi^{\otimes k}\|_{\e \to \pi} \leq \| \phi_1^{\otimes k}\|_{\e \to h}\| \phi_2^{\otimes k}\|_{h \to \pi} \]
and therefore, using Theorem \ref{thm:Main1} and Corollary \ref{cor:Main12}, we find that
\[
\limsup_{k \to \iy} \|\phi^{\otimes k}\|^{1/k}_{\e \to \pi} \leq \pi_2(\phi_1) \pi_2(\phi_2^*).
\]
Taking the infimum over all such factorizations shows that 
\begin{equation}\label{eq:tau_ep_upperbound}
\limsup_{k \to \iy} \|\phi^{\otimes k}\|^{1/k}_{\e \to \pi} \leq \gamma_2^*(\phi).
\end{equation}
To prove the reverse inequality in \eqref{eq:single-letter-formula}, we start by proving it in special case of operators between finite-dimensional Hilbert spaces. This was proven in \cite{paperB}, but for completeness we include a proof.

\begin{thm}\label{thm:identity}
Let $H$ denote a finite-dimensional Hilbert space and $\ident:H\ra H$ the identity map. We have
\[
\lim_{k\ra \infty}\|\ident^{\otimes k}\|^{1/k}_{\epsilon\ra \pi} = \dim(H) =\gamma_2^*(\ident).
\]
\end{thm}

\begin{proof}
For any $z\in H^{\otimes_h k}$, we have
\[
\|z\|^2_{h}=|\braket{z}{z}|\leq \|z\|_{\epsilon}\|z\|_{\pi},
\]
by duality of the injective and projective norms in finite dimensions. This shows 
\[
\|\ident^{\otimes k}\|_{\epsilon\ra \pi} \geq \|\ident^{\otimes k}\|_{\epsilon\ra h}^2 .
\]
By Theorem \ref{thm:Main1}, we conclude that
\[
\liminf_{k\ra \infty}\|\ident^{\otimes k}\|^{1/k}_{\epsilon\ra \pi} \geq \pi_2(\ident)^2 = \mathrm{hs}(\ident)^2 = \dim(H).
\]
Since $\gamma_2^*(\ident)=\dim(H)$, the theorem follows from \eqref{eq:tau_ep_upperbound}.
\end{proof}

Next, we prove our second main result:

\begin{thm}\label{thm:Main2}
Let $X$ and $Y$ denote finite-dimensional normed spaces. For any operator $\phi : X \to Y$ we have
\[ \lim_{k \to \iy} \|\phi^{\otimes k}\|_{\e \to \pi}^{1/k} =\gamma_2^*(\phi). \]
\end{thm}

\begin{proof}
Consider a finite-dimensional Hilbert space $H$ and operators $\alpha:Y\ra H$ and $\beta:H\ra X$ with $\|\alpha\|,\|\beta\|\leq 1$. Let $\mathrm{d}u$ denote the Haar measure on the orthogonal group $G=O(H)$ (or the unitary group $G=U(H)$ in the complex case), and set
\[
\psi := \int_{G} u\beta\phi\alpha u^{-1} \, \mathrm{d}u .
\]
For any $k \geq 1$, it follows from the triangle inequality and the metric mapping property that $\|\psi^{\otimes k}\|_{\e \to \pi} \leq \|\phi^{\otimes k}\|_{\e \to \pi}$. Moreover, by a standard computation 
\[
\psi =\frac{\tr(\beta\phi\alpha)}{\dim(H)}\ident_H .
\]
Using Theorem \ref{thm:identity} we find that
\[
|\tr (\beta\phi\alpha)|=\liminf_{k\ra \infty}\|\psi^{\otimes k}\|^{1/k}_{\e \to \pi}\leq \liminf_{k\ra \infty}\|\phi^{\otimes k}\|^{1/k}_{\e \to \pi}.
\]
Taking the supremum over finite-dimensional Hilbert spaces $H$ and operators $\alpha,\beta$  as above, shows that 
\[
\gamma_2^*(\phi)\leq \liminf_{k\ra \infty}\|\phi^{\otimes k}\|^{1/k}_{\e \to \pi}.
\]
Now, the theorem follows from \eqref{eq:tau_ep_upperbound}.
\end{proof}

\section{Extension to infinite-dimensional spaces}
\label{sec:infinite-dimension}

Given Banach spaces $X_1$ and $X_2$, we denote by $X_1 \bar{\otimes}_\e X_2$ the completion of $X_1 {\otimes}_\e X_2$, and similarly for other tensor norms. A bounded operator $\phi$ from $X_1 \otimes_\e X_2$ to a Banach space extends uniquely to an operator on $X_1 \bar{\otimes}_\e X_2$ which we also denote by $\phi$. By repeating the argument given for Proposition \ref{prop:pi2-tensorstable}, we find the following:

 \begin{prop} \label{prop:pi2-tensorstable-infdim}
 Let $X_1$, $X_2$ be Banach spaces, $H_1$, $H_2$ be Hilbert spaces and $\phi_1:X_1 \to H_1$ and $\phi_2 : X_2 \to H_2$ be $2$-summing operators. Then the operator
 \[
 \phi_1 \otimes \phi_2: X_1 \bar\otimes_\epsilon X_2\ra H_1\bar\otimes_h H_2
 \]
 is $2$-summing and satisfies $\pi_2(\phi_1 \otimes \phi_2)=\pi_2(\phi_1)\pi_2(\phi_2)$.
 \end{prop}

Let $\phi:X\ra H$ denote a bounded operator from a Banach space to a Hilbert space. Similar to \eqref{eq:tau_eh_upperbound} we obtain
\begin{equation}\label{equ:UpperBoundPi2Inf} \limsup_{k \to \iy} \|\phi^{\otimes k}\|^{1/k}_{\e \to h} \leq \mathrm{hs}(\phi),\end{equation}
from Proposition \ref{prop:pi2-tensorstable-infdim}. To show the reverse inequality we start in the case of Hilbert spaces.

\begin{prop} \label{p:Lemma2EpsilonH-infdim}
For every Hilbert space $H$ and every bounded operator $\phi : H \to H$, we have
\[ \lim_{k \to \iy} \|\phi^{\otimes k}\|_{\e \to h}^{1/k} = \mathrm{hs}(\phi) = \pi_2(\phi). \]
In particular, the limit is finite if and only if $\phi$ is Hilbert-Schmidt.
\end{prop}

\begin{proof}
If $H$ is finite-dimensional, this is the content of Proposition \ref{p:Lemma2EpsilonH}. For a finite-dimensional subspace $K \subset H$, we denote by $\phi_K$ the restriction of $\phi$ to ~$K$. By the metric mapping property (applied to the canonical inclusion) we have $\|\phi_K^{\otimes k}\|_{\e \to h} \leq \|\phi^{\otimes k}\|_{\e \to h}$ for every $k$ and therefore
\[ \liminf_{k \to \iy} \|\phi^{\otimes k}\|^{1/k}_{\e \to h} \geq 
\liminf_{k \to \iy} \|\phi_K^{\otimes k}\|^{1/k}_{\e \to h} = \mathrm{hs}(\phi_K), \]
where we used Proposition \ref{p:Lemma2EpsilonH}. Since 
\[ \label{eq:HS-finitedimension} \mathrm{hs}(\phi) = \sup \{ \mathrm{hs}(\phi_K) ~:~ K \textnormal{ finite-dimensional subspace of }H \},\]
we conclude that
\[  \liminf_{k \to \iy} \|\phi^{\otimes k}\|^{1/k}_{\e \to h} \geq \mathrm{hs}(\phi).
\]
Together with \eqref{equ:UpperBoundPi2Inf} this completes the proof.
\end{proof}

Repeating the argument given for Theorem \ref{thm:Main1} and using Proposition \ref{p:Lemma2EpsilonH-infdim} shows the following: 

\begin{thm} \label{thm:Main1-infdim}
Let $X$ denote a Banach space and $H$ a Hilbert space. For every bounded operator $\phi : X \to H$, we have
\[ \lim_{k \to \iy} \|\phi^{\otimes k}\|^{1/k}_{\e \to h} = \pi_2(\phi). \]
In particular, the limit is finite if and only if $\phi$ is $2$-summing.
\end{thm}

Next, we will prove a general version of the upper bound in Corollary \ref{cor:Main12}. The argument requires some care. 

\begin{thm} \label{thm:Main12-infdim}
Let $H$ denote a Hilbert space and $Y$ a Banach space. For every bounded operator $\phi : H \to Y$, we have
\[ \limsup_{k \to \iy} \|\phi^{\otimes k}\|_{h \to \pi}^{1/k} \leq \pi_2(\phi^*). \]
\end{thm}

\begin{proof}
Since $H$ is reflexive, we may identify $H$ and $H^{**}$. Denote by $j_Y : Y \to Y^{**}$ the canonical embedding. We use the following property of the projective tensor product (see \cite[Proposition 3.9]{defant1992tensor}): for any Banach space $Z$, the map
\[ j_Y \otimes \ident_Z : Y \otimes_\pi Z \to Y^{**} \otimes_\pi Z\]
is isometric. It follows by induction that for every $k \geq 1$, the map
\[ j_Y^{\otimes k} : Y^{\otimes_\pi k} \to (Y^{**})^{\otimes_\pi k}\]
is isometric. We have $\phi^{**}= j_Y \phi$ and therefore
\[ \|(\phi^{**})^{\otimes k}\|_{h \to \pi} = \| j_Y^{\otimes k} \phi^{\otimes k}\|_{h \to \pi} = \|\phi^{\otimes k}\|_{h \to \pi}.\]
Consider a factorization of $\phi^* : Y^* \to H^*$ as 
\[ \phi^* : Y^* \stackrel{\alpha}{\longrightarrow} L^{\iy}(\mu) \stackrel{i_\mu}{\longrightarrow} L^2(\mu) \stackrel{\beta}{\longrightarrow} H^* , \]
with bounded operators $\alpha$ and $\beta$ and the canonical embedding $i_\mu$. Taking adjoints, we obtain a factorization of $\phi^{**}$ as
\[ \phi^{**} : H \stackrel{\beta^*}{\longrightarrow} L^2(\mu)^* \stackrel{i_\mu^*}{\longrightarrow} L^\infty(\mu)^* \stackrel{\alpha^*}{\longrightarrow} Y^{**}\]
The range of the map $i_\mu^*$ is contained in $L^1(\mu)$ seen as a subspace of $L^{\infty}(\mu)^*$. Let~$\alpha^\circ$ be the restriction of $\alpha^*$ to $L^1(\mu)$. For every $k \geq 1$, set $\mu_k = \mu^{\otimes k}$ and identify $L^2(\mu)^{\bar{\otimes}_h k}$ with $L^2(\mu_k)$ and $L^1(\mu)^{\bar{\otimes}_\pi k}$ with $L^1(\mu_k)$ (using \cite[Example~2.19]{Ryan02} and Fubini's theorem. By the metric mapping property,
\[ \| (\beta^*)^{\otimes k} : H^{\otimes k} \to L^2(\mu_k) \| = \|\beta\|^k\]
\[ \| (\alpha^\circ)^{\otimes k} : L^1(\mu_k) \to (Y^{**})^{\otimes_\pi k}\| = \|\alpha^\circ\|^k \leq \|\alpha\|^k\]
It follows that
\[ \|\phi^{\otimes k}\|^{1/k}_{h \to \pi} = \| (\phi^{**})^{\otimes k} \|^{1/k}_{h \to \pi} \leq \| \alpha \| \cdot \|\beta\|. \]
Taking the lower bound over factorizations shows that
\[ \|\phi^{\otimes k}\|^{1/k}_{h \to \pi} \leq \pi_2(\phi^*) \]
and the result follows.
\end{proof}

Finally, we can show the general version of Theorem \ref{thm:Main2}.

\begin{thm}\label{thm:Main2Inf}
Let $X$ and $Y$ denote Banach spaces. For any bounded operator $\phi : X \to Y$ we have
\[ \lim_{k \to \iy} \|\phi^{\otimes k}\|_{\e \to \pi}^{1/k} =\gamma_2^*(\phi). \]
In particular, the limit is finite if and only if $\phi$ is $2$-dominated.
\end{thm}

\begin{proof}
Let $\phi:X\ra Y$ be $2$-dominated with a factorization $\phi=\phi_2\phi_1$ for a Hilbert space $H$ and operators $\phi_1:X\ra H$ and $\phi_2:H\ra Y$, where $\phi_1$ and $\phi^*_2$ are $2$-summing. We have
\[ \| \phi^{\otimes k}\|_{\e \to \pi} \leq \| \phi_1^{\otimes k}\|_{\e \to h}\| \phi_2^{\otimes k}\|_{h \to \pi} \]
and therefore, using Theorem \ref{thm:Main1-infdim} and Theorem \ref{thm:Main12-infdim}, we find that
\[
\limsup_{k \to \iy} \|\phi^{\otimes k}\|^{1/k}_{\e \to \pi} \leq \pi_2(\phi_1) \pi_2(\phi_2^*).
\]
Taking the infimum over all such factorizations shows that 
\begin{equation}\label{eq:tau_ep_upperbound-infdim}
\limsup_{k \to \iy} \|\phi^{\otimes k}\|^{1/k}_{\e \to \pi} \leq \gamma_2^*(\phi).
\end{equation}
To prove the other direction, consider finite-dimensional normed spaces $E$ and $F$, a finite-dimensional Hilbert space $H$, operators $\alpha:E\ra X$, $\beta:Y\ra F$, $\psi_1:F\ra H$ and $\psi_2:H\ra E$ all with norm less than $1$. As in the proof of Theorem \ref{thm:Main2} we find that
\[
|\tr(\psi_1\beta\phi\alpha \psi_2)| \leq \liminf_{k\ra \infty}\|\phi^{\otimes k}\|^{1/k}_{\e \to \pi}.
\]
Using \eqref{eq:trace-duality-infdim} shows that  
\[
\gamma_2^*(\phi)\leq \liminf_{k\ra \infty}\|\phi^{\otimes k}\|^{1/k}_{\e \to \pi},
\]
completing the proof.
\end{proof}

We point out the following simple consequence of Theorem \ref{thm:Main2Inf}. 

\begin{cor}
Let $X$ and $Y$ denote Banach spaces. If $\phi:X\ra Y$ is $2$-dominated, then the operators
 \[
 \phi^{\otimes k}: X^{\otimes_\epsilon k}\ra Y^{\otimes_\pi k}
 \]
 are $2$-dominated for every $k\in\N$, and satisfy $\gamma_2^*(\phi^{\otimes k})=\gamma_2^*(\phi)^k$.
\end{cor}

\section{Discussion}

\subsection{Tensor radii} We have established that the relation
\begin{equation} \lim_{k \to \iy} \left\| \phi^{\otimes k}  \right\|_{\e \to \pi}^{1/k} = \gamma_2^*(\phi), \end{equation}
holds for any bounded map $\phi : X \to Y$ between Banach spaces. In \cite{paperB}, the quantity on the right-hand side was studied under the name of \emph{tensor radius}. 

We now focus on the case of the identity map $\ident_X$ on an $n$-dimensional normed space $X$. The quantity $\mathrm{d}_X = \gamma_2(\ident_X)$ equals the Banach--Mazur distance between~$X$ and the $n$-dimensional Euclidean space. The inequality
\begin{equation} \label{eq:TensorRadiusLowerBound} \mathrm{d}_X \cdot \gamma_2^*(\ident_X) \geq n \end{equation}
then follows from the trace duality. If the space $X$ has enough symmetries, it follows from Lewis' theorem (see \cite[Proposition 16.1]{tomczak1989banach}) that \eqref{eq:TensorRadiusLowerBound} is an equality, recovering a result from \cite{paperB}. 

It is also easy to give examples of spaces where \eqref{eq:TensorRadiusLowerBound} is not an equality, answering a question from \cite{paperB}. Consider the space $X = X_1 \oplus_2 X_2$ where $X_1 = \ell_1^m$ and $X_2 = \ell_2^m$. We have $\dim(X) = 2m$, $\mathrm{d}_{X} \geq \mathrm{d}_{X_1} = \sqrt{m}$ and by ideal property for $\gamma_2^*$
\[ \gamma_2^*(\ident_X) \geq \gamma_2^*(\ident_{X_2}) = m.\]
Consequently, the inequality in \eqref{eq:TensorRadiusLowerBound} is strict for $m$ large enough.

\subsection{The Hilbert space factorization property}

We say that a pair $(X,Y)$ of Banach spaces has the \emph{Hilbert space factorization property} (HFP) if every bounded operator $\phi : X \to Y$ satisfies $\gamma_2(\phi) = \|\phi\|$ and ask which pairs have the HFP. Previously, the isomorphic version of this question, i.e., determining pairs of spaces for which these norms are equivalent, has received a lot of attention (see for example~\cite{Pisier12}). By trace duality, the pair $(X,Y)$ has the HFP if and only if every nuclear operator $\psi : Y \to X$ satisfies $\gamma_2^*(\psi) = \|\psi\|_N$, the nuclear norm. This property was studied in \cite{paperB} under the name \emph{nuclear tensorization property} and motivated by analogous questions in quantum information theory. 

If either $X$ or $Y$ is a Hilbert space, it is obvious that the pair $(X,Y)$ has the~HFP. In the complex case, there is another example.

\begin{prop}
The pair $(\ell^2_\infty(\C),\ell^2_1(\C))$ has the HFP. 
\end{prop}

\begin{proof}
Let $K_G(m,n)$ be the smallest constant $C$ such that the inequality $\gamma_2^*(\phi) \leq C \|\phi\|$ holds for every operator $\phi : \ell^m_\infty(\C) \to \ell^n_1(\C)$. Note that the supremum of $K_G(m,n)$ over integer $m, n$ is the usual Grothendieck constant. It is known (see~\cite{tonge1985complex,davie2006matrix, jameson84}) that $K_G(2,2)=1$ and therefore, when $m=n=2$, 
\[ \|\phi\| \leq \gamma_2(\phi) \leq \gamma_2^*(\phi) \leq \|\phi\|\]
as needed.
\end{proof}

In the real case, we conjecture that there is no non-trivial pair with the HFP.

\begin{conj} \label{conj:NTP-reals}
A pair $(X,Y)$ of real Banach spaces has the HFP if and only if $X$ or $Y$ is a Hilbert space.
\end{conj}

The pair $(\ell_{\iy}^2(\C),\ell_1^2(\C))$, when these spaces are seen as normed spaces over the reals, fails the HFP. This fact, which answers a question from \cite{paperB}), can be seen using the following proposition.

\begin{prop}
If $X$ and $Y$ are non-strictly convex real Banach spaces, then the pair $(X,Y)$ fails the HFP.
\end{prop}

\begin{proof}
There exist vectors $x_1 \neq x_2$ in $B_X$ and $y_1 \neq y_2$ in $B_Y$ such that
\[ \|\lambda x_1 + (1-\lambda)x_2 \| = \|\lambda y_1 + (1-\lambda)y_2 \| =1 ,\]
for every $0 \leq \lambda \leq 1$. Consider $f,g \in B_{X^*}$ such that $f(\frac{x_1+x_2}{2})=1$ and $g(x_1-x_2)\neq 0$. In particular, we have $f(x_1)=f(x_2)=1$. We define an operator $\phi : X \to Y$ by
\[ \phi(x) = \frac{f(x)+g(x)}{2}y_1 + \frac{f(x)-g(x)}{2}y_2.\]
It is easy to verify that $\|\phi(x_1+x_2)\|=2$. For any $x \in B_X$, we have
\begin{equation} \label{eq:nonstrictlyconvex} \|\phi(x)\| \leq \left| \frac{f(x)+g(x)}{2} \right| + \left| \frac{f(x)-g(x)}{2} \right| = \max(|f(x)|,|g(x)|) \leq 1\end{equation}
and thus $\|\phi\|=1$. Let $\e = \|\phi(x_1-x_2)\|>0$ and consider a factorization $\phi = \beta \alpha$ with a Hilbert space $H$ and operators $\alpha : X \to H$ and $\beta : H \to Y$.  We have, using the parallelogram identity in $H$,
\begin{eqnarray*}
\e^2 & \leq & \|\beta\|^2 \|\alpha(x_1-x_2)\|^2 \\
     & = & \|\beta\|^2 \left( 2 \|\alpha(x_1)\|^2 + 2\|\alpha(x_2)\|^2 - \|\alpha(x_1+x_2) \|^2 \right) \\
     & \leq & \|\beta\|^2 \left( 2\|\alpha\|^2 + 2\|\alpha\|^2 - 4 \|\beta\|^{-2} \right).
\end{eqnarray*}
Taking the infimum over factorizations gives $\gamma_2(\phi) \geq \sqrt{1+\e^2/4} > \|\phi\|$.
\end{proof}

We point out that the previous proof fails in the complex case as the equality in \eqref{eq:nonstrictlyconvex} is no longer valid. Our next goal is to reduce Conjecture \ref{conj:NTP-reals} to a question about $2$-dimensional spaces. We first prove a heredity property for the HFP.

\begin{prop} \label{prop:NTP-strong-inheritance}
Assume that $X$, $Y$ are Banach spaces such that the pair $(X,Y)$ has the HFP. Then, for every subspace $E$ of $Y$ and any quotient $X/F$ of $X$, the pair $(X/F,E)$ has the HFP.
\end{prop}

\begin{proof}
Consider a linear operator $\phi:X/F\ra E$. Denoting by $j:E\ra Y$ the inclusion and by $q:X\ra X/F$ the quotient map, we have $j\phi q:X\ra Y$ and $\gamma_2(j\phi q)=\|j\phi q\|$. By~\cite[Proposition 7.3]{DJT95} we find that 
\[
\gamma_2(\phi) =\gamma_2(j\phi q)=\|j\phi q\|\leq \|\phi\| .
\]
The other inequality is clear.
\end{proof}

It is a classical fact (see~\cite{Amir86}) that a normed space $X$ is Euclidean if and only if all its $2$-dimensional subspaces are Euclidean, and also if and only if all its $2$-dimensional quotients are Euclidean. Using Proposition \ref{prop:NTP-strong-inheritance}, this shows that in order to prove Conjecture \ref{conj:NTP-reals} it suffices to consider the case $\dim X = \dim Y =~2$. This allows to formulate a concrete geometric question equivalent to Conjecture~\ref{conj:NTP-reals}.

\begin{conj} \label{conj:dim2}
Let $K$ and $L$ be centrally symmetric convex bodies in $\R^2$. Assume that for every linear map $T$ such that $T(K) \subset L$, there is an ellipse $\mathcal{E} \subset \R^2$ such that $T(K) \subset \mathcal{E} \subset L$. Then $K$ or $L$ is an ellipse.
\end{conj}

\section*{Acknowledgments}

We thank Roy Araiza and Mikael de la Salle for helpful discussions. The authors acknowledge support of the Institut Henri Poincaré (UAR 839 CNRS-Sorbonne Université), and LabEx CARMIN (ANR-10-LABX-59-01). GA was supported in part by ANR (France) under the grant ESQuisses (ANR-20-CE47-0014-01). AMH acknowledges funding from The Research Council of Norway (project 324944). 

\appendix

\section{The random tensors argument}

A random variable $g$, with values in a finite-dimensional Hilbert space $H$, is said to be a \emph{standard Gaussian vector} if for any orthonormal basis $(e_i)$ of $H$ the random variables $\scalar{g}{e_i}$ are independent with a standard normal distribution. We use the following observation.

\begin{lem} \label{lemma:random-tensor-HS}
Let $g$ be a standard Gaussian vector in a finite-dimensional Hilbert space $H$ and $\phi : H \to H$. Then, we have
\[ \E \|\phi(g)\|^2 = \mathrm{hs}(\phi)^2 \]
and
\[ \E \|\phi(g)\| \geq \frac{1}{\sqrt{3}} \mathrm{hs}(\phi). \]
\end{lem}

\begin{proof}
By the singular value decomposition, we may assume that $\phi(e_i)=\lambda_i e_i$ for some orthonormal basis $\lset e_i\rset$ of $H$ and with all $\lambda_i$ non-negative. Let $g_i=\scalar{g}{e_i}$ for every $i$, and observe that $\E |g_i|^2 |g_j|^2$ is at most $3$. We compute
\[ \E \|\phi(g)\|^2 = \E \sum \lambda_i^2  |g_i|^2 = \sum \lambda_i^2 = \mathrm{hs}(\phi)^2\]
and
\[ \E \|\phi(g)\|^4 = \E \sum_{i,j} \lambda_i^2\lambda_j^2 |g_i|^2|g_j|^2 \leq 3 \sum_{i,j} \lambda_i^2\lambda_j^2 = 3\cdot\mathrm{hs}(\phi)^4.\]
We conclude by using H\"older's inequality in the form $\E |Z| \geq \frac{(\E |Z|^2)^{3/2}}{(\E |Z|^4)^{1/2}}$ for the random variable $Z=\|\phi(g)\|$.
\end{proof}

For any $k \geq 2$, we choose $g_k$ to be a standard Gaussian vector in the Hilbert space $H^{\otimes_h k}$. We have
\begin{equation} \label{eq:randomtensor1}
\mathbf{E}\|g_k\|_{\e} \leq C_n \sqrt{k\log(k)} ,
\end{equation}
by a net argument (see \cite[Lemma 4.3]{paperB}, which is stated in the spherical rather than Gaussian normalization) and where $C_n$ is a constant depending only on $n = \dim H$. In the real case, the Chevet inequality would give $\mathbf{E} \|g_k\|_\e \leq k\sqrt{n}$ which is good enough for our purposes. By Lemma \ref{lemma:random-tensor-HS}, we have
\begin{equation} \label{eq:randomtensor2}
\mathbf{E}\|\phi^{\otimes k}(g_k)\| \geq \frac{1}{\sqrt{3}} \mathrm{hs}(\phi^{\otimes k}) = \frac{1}{\sqrt{3}} \mathrm{hs}(\phi)^k .
\end{equation}
By linearity of expectation, it follows from \eqref{eq:randomtensor1} and \eqref{eq:randomtensor2} that there exists $z_k \in H^{\otimes k}$ such that
\[
\|z_k\|_{\epsilon}\leq C_n\sqrt{k\log(k)} ,
\]
and 
\[
\|\phi^{\otimes k}(z_k)\| \geq \frac{1}{\sqrt{3}} \mathrm{hs}(\phi)^{k} .
\]

\bibliographystyle{alpha}
\bibliography{biblio.bib}

\end{document}